\documentclass[11pt, oneside,a4]{amsart}
\usepackage{amsfonts, amstext, amsmath, amsthm, amscd, amssymb}
\usepackage{manfnt,hyperref}
\usepackage{mathrsfs,stmaryrd}
\usepackage{url}
\usepackage[all]{xypic}

\usepackage{pstricks}
\usepackage{pstricks-add}
\usepackage{multido}

\usepackage{marginnote}

\usepackage{enumerate}
\usepackage{graphics, pinlabel, color}
\usepackage{comment}

\usepackage[all,graph]{xy}
\usepackage{psfrag}

\usepackage{float, placeins,  longtable}

\renewcommand{\setminus}{{\smallsetminus}}
\newcommand{\tmfrac}[2]{\mbox{\large$\frac{#1}{#2}$}} 

\newcommand{\bp}{\begin{pmatrix}}
\newcommand{\ep}{\end{pmatrix}}
\newcommand{\be}{\begin{equation}}
\newcommand{\ee}{\end{equation}}
\newcommand{\ol}[1]{\overline{#1}}

\numberwithin{equation}{section}

\theoremstyle{plain}
\newtheorem{theorem}[equation]{Theorem}
\newtheorem{lemma}[equation]{Lemma}
\newtheorem{proposition}[equation]{Proposition}

\newtheorem{corollary}[equation]{Corollary}

\newtheorem*{claim*}{Claim}

\theoremstyle{definition}

\newtheorem{ex}[equation]{Example}

\numberwithin{equation}{section}

 \newtheoremstyle{TheoremNum}
        {}{}              
        {\itshape}                      
        {}                              
        {\bfseries}                     
        {.}                             
        { }                             
        {\thmname{#1}\thmnote{ \bfseries #3}}
\theoremstyle{TheoremNum}

\def\Z{\mathbb Z}
\def\R{\mathbb R}
\def\Q{\mathbb Q}

\def\L{\Lambda}

\def\wt#1{\widetilde{#1}}

\def\sm{\setminus}

\def\toiso{\xrightarrow{\simeq}}

\def\ll{\langle}

\def\rr{\rangle}

\def\zt{\Z[t^{\pm 1}]}

\def\bp{\begin{pmatrix}}
\def\ep{\end{pmatrix}}
\def\ba{\begin{array}}
\def\ea{\end{array}}
\def\bn{\begin{enumerate}}
\def\en{\end{enumerate}}

\def\op{\operatorname}
\def\BS{\op{BS}}
\def\PD{\op{PD}}
\def\ev{\op{ev}}

\DeclareMathOperator\Hom{Hom}

\DeclareMathOperator\Id{Id}
\DeclareMathOperator\id{Id}

\DeclareMathOperator\Bl{Bl}
\DeclareMathOperator\im{im}

\DeclareMathOperator\coker{coker}

\def\wti{\widetilde}

\def\hom{\op{Hom}}
\def\ol{\overline}
\def\PD{\op{PD}}
\def\ev{\op{ev}}

\begin{document}

\title{A calculation of Blanchfield pairings of 3-manifolds and knots}

\author{Stefan Friedl}
\address{Fakult\"at f\"ur Mathematik\\ Universit\"at Regensburg\\   Germany}
\email{sfriedl@gmail.com}

\author{Mark Powell}
\address{
D\'epartement de Math\'ematiques, Universit\'e du Qu\'ebec \`a Montr\'eal, QC, Canada}
\email{mark@cirget.ca}


\def\subjclassname{\textup{2010} Mathematics Subject Classification}
\expandafter\let\csname subjclassname@1991\endcsname=\subjclassname
\expandafter\let\csname subjclassname@2000\endcsname=\subjclassname

\begin{abstract}
We calculate Blanchfield pairings of $3$-manifolds. In particular, we give a formula for the Blanchfield pairing of a fibred $3$-manifold and we give a new proof that the Blanchfield pairing of a knot can be expressed in terms of a Seifert matrix.
\end{abstract}
\maketitle

\section{Introduction}

Let $X$ be a 3-manifold. Throughout the paper we assume that all manifolds are compact, connected and oriented, and we assume that all 3-manifolds are either closed or that they have toroidal boundary.
Let $\phi \in H^1(X;\Z) =[X,S^1]$ be a nontrivial primitive cohomology class.  The map $\phi$ gives rise to a homomorphism $\pi_1(X) \to \Z=\ll t\rr$.  We write $\Lambda=\zt$ and we denote the infinite cyclic cover of $X$ corresponding to $\phi$ by $\wti{X}$. Thus $\phi$ determines an \emph{Alexander module} $H_1(\wti{X};\Z)=H_1(X;\L)$ of $X$.  We suppose that $H_1(X;\L)$ is $\L$-torsion; that is, $H_1(X;Q)=0$, where $Q=\Q(t)$ is the quotient field of~$\Lambda$.

In 1957, Blanchfield~\cite{Bl57} introduced a pairing on the Alexander module $H_1(X;\Lambda)$ that takes values in $Q/\Lambda$. More precisely, Blanchfield showed that there exists a pairing
\[ \Bl\colon H_1(X;\Lambda)\times H_1(X;\Lambda)\to Q/\Lambda,\]
which is sesquilinear i.e.\ linear over $\Lambda$ in the first variable and conjugate-linear over $\Lambda$ in the second variable.   We refer to this pairing as the \emph{Blanchfield pairing $\Bl$ of $(X,\phi)$}.


It follows from our assumption that $H_1(X;Q)=0$ and that $\phi$ is primitive, together with~\cite[Proposition~3.4]{FK06}, that there exists a connected, properly embedded surface $F$ dual to $\phi$.
Let $F \times [-1,1]$ be a thickening and let $Y:= X \sm (F \times (-1,1))$ be the complement of the thickened surface.  Let $\iota_{\pm} \colon H_1(F;\Z) \toiso H_1(F \times \{\pm 1\}) \to H_1(Y;\Z)$ be the two inclusion induced maps. Our assumption that $H_1(X;Q)=0$ implies, by a standard Mayer--Vietoris argument, that
$\iota_+-t^{-1}\iota_-\colon Q\otimes H_1(F;\Z) \to Q\otimes H_1(Q;\Z)$ is an isomorphism.
Let $\iota \colon H_1(F;\Z) \to H_1(\wti{X};\Z)=H_1(X;\Lambda)$ be the map that is induced by a lift of $F$ to the infinite cyclic cover $\wti{X}$.
Our main technical theorem obtains the following expression for the Blanchfield pairing of any two elements of $H_1(X;\L)$ lying in the image of~$\iota$.

\begin{theorem}\label{thm:computebl-intro}
Let $X$ be a  $3$-manifold  together with a primitive class $\phi \in H^1(X;\Z)$, and suppose that $H_1(X;Q)=0$.
For any connected surface $F$ dual to $\phi$ and for any $v,w\in \L \otimes  H_1(F;\Z)$ we have
\[ \op{Bl}(\iota v,\iota w)=-(\iota_+-t^{-1}\iota_-)^{-1}(\iota_+(v))\underset{F}{\cdot} w.\]
Here $\iota_+(v)$ lies in $\Lambda\otimes H_1(Y;\Z)\subset Q\otimes H_1(Y;\Z)$,  the map
$\iota_+-t^{-1}\iota_-\colon Q\otimes H_1(F;\Z) \to Q\otimes H_1(Y;\Z)$ is the aforementioned isomorphism and $\underset{F}\cdot$ denotes the sesquilinear intersection pairing $Q\otimes H_1(F;\Z)\times Q\otimes H_1(F;\Z)\to Q$.
\end{theorem}

\subsection{Fibred 3-manifolds}

For our first application of Theorem~\ref{thm:computebl-intro} we compute the Blanchfield pairing of a fibred 3-manifold.
First we fix some notation. Let $F$ be a surface and let $\varphi\colon F\to F$ be a self-homeomorphism.
We denote the corresponding mapping torus by $M=M(F,\varphi)=F\times [0,1]/(x,0)\sim (\varphi(x),1)$, which we endow with the usual product orientation. We refer to $\phi\colon \pi_1(M)\to \pi_1([0,1]/0\sim 1)=\Z$ as the canonical epimorphism.  We remark that $H_1(X;\L)$ is $\L$-torsion, since the infinite cyclic cover $\wt{X}$ of $X$ is homeomorphic to $F \times \R$, and hence the homology $H_1(\wt{X};\Z) =H_1(X;\L)$ is finitely generated over~$\Z$.

\begin{corollary}\label{cor:fib}
Let $F$ be a  surface and let $\varphi\colon F\to F$ be a self-homeomorphism.
Denote the mapping torus by $M=M(F,\varphi)$ and denote the canonical epimorphism by $\phi \colon \pi_1(M) \to \Z$.  Pick a basis $c_1,\dots,c_k$ for $H_1(F;\Z)$. With respect to this basis, let $J$ be the matrix representing the intersection pairing on $H_1(F;\Z)$, and let $P$ be the matrix representing the monodromy $\varphi_*\colon H_1(F;\Z)\to H_1(F;\Z)$. Then the Blanchfield pairing of $(M,\phi)$ is isomorphic to
\[ \ba{rcl} \L^k/(tP-\id)\times \L^k/(tP-\id)&\to &Q/\Lambda\\
(v,w)&\mapsto & v^T J(t^{-1}P-\id)^{-1}\ol{w}.\ea\]
\end{corollary}

Here it is an entertaining exercise to verify that the formula in the corollary makes sense. In the verification one needs that $\varphi$ preserves the intersection pairing on $F$ i.e.\ that $P^T JP=J$.

\subsection{The Blanchfield pairing of a knot}

Another important special case arises from knot theory.
Given an oriented knot $K\subset S^3$ we write $X_K:=S^3\sm \nu K$, where $\nu K$ denotes an open tubular neighbourhood around $K$.  A generator $\phi \in H^1(X_K;\Z)$ gives rise to the Alexander module $H_1(X_K;\L)$, and this is $\L$-torsion for any~$K$.  Let $F$ be a genus $g$ Seifert surface for $K$ and let $A$ be a $2g\times 2g$ Seifert matrix for $K$ with respect to a basis for $H_1(F;\Z)$.  The matrices $A$ and $A^T$ represent $\iota_+$ and $\iota_-$ respectively (see the statement of Theorem~\ref{thm:blanchfield-k-intro} below for more details).
It is well known, and straightforward to show, that the Alexander module $H_1(X_K;\Lambda)$ is isomorphic to $\L^{2g}/(tA-A^T)$.
It is also well known, but less straightforward to show, that the Blanchfield pairing $\Bl_K$ can be expressed in terms of a Seifert matrix.
After proving Theorem~\ref{thm:computebl-intro}, we will deduce the following formula for the Blanchfield pairing of a knot.

\begin{theorem}\label{mainthm}
Let  $K\subset S^3$ be an oriented knot.
Let $A$ be a Seifert matrix for $K$ of size $2g$.
The Blanchfield pairing of $K$ is isometric to the pairing
\[ \ba{rcl} \Lambda^{2g}/(tA-A^T) \times \Lambda^{2g}/(tA-A^T)&\to & Q/\Lambda\\
(v,w)&\mapsto& {v}^T(t-1)(A - tA^T)^{-1}\ol{w}.\ea\]
\end{theorem}

In this instance we also encourage the reader to verify that the pairing in Theorem~\ref{mainthm} is indeed well-defined and sesquilinear.
In particular it is worth noting that the discrepancy between $tA-A^T$ in the description of the module and $A-tA^T$  in the definition of the pairing is deliberate, and is in fact necessary in order for the pairing to be well-defined.

Before we discuss the history of Theorem~\ref{mainthm} we state a more precise, albeit less readable version.

\begin{theorem}\label{thm:blanchfield-k-intro}
Let  $K\subset S^3$ be an oriented knot.
Let $F$ be a Seifert surface for the knot $K$ and let  $\{d_1,\dots,d_{2g}\}$ be a basis for $H_1(F;\Z)$.
Denote the Seifert matrix of $F$ corresponding to the basis $\{d_1,\dots,d_{2g}\}$  by $A$, and
 write  $Y:= X_K \sm (F \times(-1,1))$.  Let $\{e_1,\dots,e_{2g}\}$ be the basis of $H_1(Y;\Z)$ that is  Alexander dual to the basis $\{d_1,\dots,d_{2g}\}$, that is $\op{lk}(d_i,e_j)=\delta_{ij}$.
Then the homomorphism
\[ \ba{rcl} \Phi\colon \Lambda^{2g}&\to& H_1(X_K;\Lambda)\\
  (p_1,\dots,p_{2g})&\mapsto & \sum\limits_{i=1}^{2g}p_ie_i \ea\]
induces an isomorphism
\[  \Phi\colon \Lambda^{2g}/(tA-A^T)\Lambda^{2g}\xrightarrow{\cong} H_1(X_K;\Lambda),\]
which in turn gives rise to a commutative diagram
\[ \xymatrix@C4cm{\Lambda^{2g}/(tA-A^T)\Lambda^{2g}\times \Lambda^{2g}/(tA-A^T)\Lambda^{2g}\ar[r]^-{(v,w)\mapsto v^T(t-1)(A - tA^T)^{-1}\ol{w}}\ar[d]^{\Phi\times \Phi}&Q/\Lambda\ar[d]^=\\
H_1(X_K;\Lambda)\times H_1(X_K;\Lambda)\ar[r]^-{(v,w)\mapsto \Bl_K(v,w)}&Q/\Lambda.}\]
\end{theorem}

Kearton~\cite{Ke75} first gave a description of the Blanchfield pairing in terms of the Seifert matrix. On \cite[page~159]{Ke75} it is stated, in term of the above notation, that for any $i,j\in \{1,\dots,2g\}$, we have
\[ \Bl_K(e_i,e_j)\,\,=\,\, (1-t)\big((tA-A^T)^{-1}\big)_{ij}.\]
Kearton \cite[p.~150]{Ke75} uses the same notion of sesquilinearity as we do, thus his formula translates into the statement that the Blanchfield form is isometric to
\[ \ba{rcl} \Lambda^{2g}/(tA-A^T) \times \Lambda^{2g}/(tA-A^T)&\to & Q/\Lambda\\
(v,w)&\mapsto& {v}^T(t-1)(tA - A^T)^{-1}\ol{w}.\ea\]
But the reader who did the above exercise will see that this form is not even well-defined.

Modulo the subtlety just discussed, the statement of Theorem~\ref{mainthm} is well known.  As well as its appearance in \cite{Ke75} discussed above, an argument was given in~\cite[Proposition~14.3]{Le77}.  An algebraic generalisation can be found in~\cite{Ra03}.
Our proof is quite different from the one in \cite{Le77}, and has the virtue of being much more amenable to generalisation. We have already seen that our proof covers the case of fibred 3-manifolds, and in a future paper we will use our approach to provide a way to calculate twisted Blanchfield pairings defined using unitary representations of $\pi_1(X_K)$. In these more general situations our proofs will become significantly harder to parse. This motivated us to write up the proof in the present more simple case separately.

The description of the Blanchfield form of a knot in Theorem~\ref{mainthm} is arguably slightly unsatisfactory, in so far as the presentation matrix for the Alexander module is not the inverse of the matrix describing the Blanchfield form.
For the sake of completeness, we give such a description of the Blanchfield form in terms of the Seifert matrix.  This first appeared in~\cite{BF15}, following~\cite{Ko89}.

Let $K$ be an oriented knot and let $A$ be any matrix of size $2k$ which is $S$-equivalent to a Seifert matrix for $K$.
Note that $A-A^T$ is skew-symmetric and it satisfies $\det(A-A^T)=(-1)^k$.
After possibly replacing $A$ by $PAP^T$ for an appropriate $P$, the following equality holds:
\[ A-A^T= \bp 0 & \id_k \\ -\id_k &0 \ep.\]
As in \cite[Section~4]{Ko89}, we define $M_K(t)$ to be the matrix
\[
\bp (1-t^{-1})^{-1}\id_k &0\\ 0&\id_k \ep A \bp \id_k&0\\0&(1-t)\id_k\ep
+\bp \id_k &0 \\ 0&(1-t^{-1})\id_k \ep A^T\bp (1-t)^{-1}\id_k &0 \\ 0&\id_k
\ep.\]
An elementary calculation shows  that  the matrix $M_K(t)$  is a hermitian matrix defined over $\L$, and that for any complex number $z \in S^1 \sm\{1\}$ the signature $\op{sign}(M_K(z))$ equals  the Levine-Tristram signature $\sigma_z(K)$.
The following is \cite[Proposition~2.1]{BF15}, which in turn was proved assuming Theorem~\ref{mainthm}.
The notion of sesquilinearity used in \cite{BF15} was different to ours, so we altered the statement accordingly, replacing $M_K(t)$ with $M_K(t^{-1})$ as the matrix for the Blanchfield pairing.

\begin{proposition}\label{prop:blakt}
Let  $K\subset S^3$ be an oriented knot and let $M_K(t)$ be as above.
The Blanchfield pairing of $K$ is isometric to the pairing
\[ \ba{rcl} \Lambda^{2k}/M_K(t) \times \Lambda^{2k}/M_K(t)&\to & Q/\Lambda\\
(v,w)&\mapsto& -{v}^T(M_K(t^{-1}))^{-1}\ol{w}.\ea\]
\end{proposition}

Proposition~\ref{prop:blakt} allows us to give a quick proof that the Blanchfield form of a knot is hermitian and nonsingular. Both statements have been well known, but it is notoriously tricky to write down a rigourous proof that the Blanchfield form is hermitian.

\begin{corollary}\label{cor:hermitian-nonsingular}
For any oriented knot $K$ the Blanchfield pairing is hermitian and nonsingular. This means that for any $v,w\in H_1(X_K;\Lambda)$ we have $\Bl(v,w)=\ol{\Bl(w,v)}$ and the map $H_1(X_K;\Lambda)\to \ol{\hom_{\Lambda}(H_1(X_K;\Lambda),Q/\L)}$ given by $v\mapsto \Bl(-,v)$ is an isomorphism.
\end{corollary}

\begin{proof}
In light of Proposition~\ref{prop:blakt} it suffices to show that if $M(t)$ is a hermitian $k\times k$-matrix over $\Lambda$ with $\det(M(t))\ne 0$, then the pairing
\[ \ba{rcl} \Lambda^{k}/M(t) \times \Lambda^{k}/M(t)&\to & Q/\Lambda\\
(v,w)&\mapsto& {v}^T(M(t^{-1}))^{-1}\ol{w}.\ea\]
is hermitian and nonsingular. To prove the first statement, suppose that $M(t) = \ol{M(t)}^T$ and observe that
\[\ol{w^TM(t^{-1})^{-1}\ol{v}} = \ol{w^TM(t^{-1})^{-1}\ol{v}}^T = v^T\ol{(M(t^{-1})^{-1})}^T\ol{w} = v^T\left(\ol{M(t^{-1})}^T\right)^{-1}\ol{w} = v^TM(t^{-1})^{-1}\ol{w}.\]

Now suppose that $\det(M(t)) \neq 0$.  Then the linking form above corresponds, under the equivalence of \cite[Proposition~3.4.1]{Ra81} between hermitian linking forms on $\L$-torsion modules and $Q$-acyclic $1$-dimensional symmetric complexes over $\L$ (see Sections 1.1 and 3.1 of \cite{Ra81} for the definitions of these notions), to the symmetric complex
\[\xymatrix @C+1.5cm { C^0  \ar[r]^{\ol{M(t)}^T=M(t)} \ar[d]_{\Id} & C^1 \ar[d]_{\Id} \\ C_1 \ar[r]^{M(t)} & C_0 }\]
with all chain and cochain groups $C_i,C^i$ isomorphic to $\L^{k}$.  Since the vertical maps are isomorphisms, this is a Poincar\'{e} complex, and so the linking form is nonsingular, by the last sentence of~\cite[Proposition~3.4.1]{Ra81}.
\end{proof}

\subsection{High dimensional knots}

  The techniques of this paper also extend fairly easily to prove the analogous formula for the Blanchfield pairing of a high odd dimensional knot $K \colon S^{2k-1} \hookrightarrow S^{2k+1}$, as was also considered in \cite{Ke75}, \cite{Le77}.  For such knots, let $X_K:= S^{2k+1} \sm \nu K$, and let $F \subset S^{2k+1}$ be a $2k$-dimensional Seifert manifold, i.e.\ $F$ is an oriented, connected submanifold of $S^{2k+1}$ with $\partial F = K$.
  For a $\L$-module $N$ let $\op{tor}N$ be the $\Z$-torsion submodule, and let $fN:=N/\op{tor}N$ be the free part. We consider the Seifert pairing
  \[ \ba{rcl} fH_k(F;\Z) \times fH_k(F;\Z) &\to& \Z\\
  (v,w)&\mapsto & \op{lk}(v,w^+),\ea\]
  where $w$ denotes a positive push-off of $w$ into the complement of $F$.
  We denote a corresponding size $2g$ Seifert matrix by $A$. The matrix $A+(-1)^kA^T$ is invertible.
  The proof of Theorems~\ref{thm:computebl-intro} and \ref{thm:blanchfield-k-intro} can be modified by setting the indices $1,2$ to be equal to $k,k+1$ respectively, and working with $\Z$-torsion free modules, to show that the Blanchfield pairing
  \[\Bl \colon fH_k(X_K;\L) \times fH_k(X_K;\L) \to Q/\L \]
is isometric to
\[ \ba{rcl} \Lambda^{2g}/(tA-A^T) \times \Lambda^{2g}/(tA-A^T)&\to & Q/\Lambda\\
(v,w)&\mapsto& {v}^T(t-1)(A - tA^T)^{-1}\ol{w}.\ea\]

\subsubsection*{Organisation of the paper}

In Section~\ref{section:definitions} we recall the definition of the Blanchfield form and we recall several tools and results needed in the study of twisted coefficients.
In Section~\ref{section:main-technical-theorem} prove Theorem~\ref{thm:computebl-intro}, which determines the Blanchfield pairing starting from a dual surface. In Section~\ref{proof-of-mainthm} we will use
Theorem~\ref{thm:computebl-intro} to deduce Theorem~\ref{thm:blanchfield-k-intro}.
Finally in Section~\ref{section:Bl-form-fibred-manifold} we apply Theorem~\ref{thm:computebl-intro} to give the proof of Corollary~\ref{cor:fib}.

\subsubsection*{Acknowledgments.}
We are grateful to Anthony Conway for giving us the impetus for the project. We also wish to thank Maciej Borodzik, David Cimasoni, Greg Friedman, Matthias Nagel and Patrick Orson for helpful conversations.
The first author was supported by the SFB 1085 ``Higher invariants'', funded by the Deutsche Forschungsgemeinschaft (DFG).  The second author is supported by the NSERC grant ``Structure of knot and link concordance.''  Part of the work of the paper was done while the second author was a visitor at the Max Planck Institute for Mathematics in Bonn.

\section{Conventions and definitions}\label{section:definitions}

\subsection{Conventions}

We will use the following conventions.
\bn
\item All manifolds are assumed to be connected, compact and oriented.
Furthermore all 3-manifolds are assumed to be  closed or  to have toroidal boundary.
\item We make the identifications
\[ H^1(X;\Z)=H_2(X,\partial X;\Z)=\op{Hom}(\pi_1(X),\Z)=\op{Hom}(\pi_1(X),\langle t\rangle).\]
\item Identify $\zt$ with the group ring of the infinite cyclic group $\Z \cong \ll t\rr$.   Denote $\L:= \zt$ and $Q:=\Q(t)$, the field of fractions of $\L$. We view $\Lambda=\zt$ as a ring with involution where, as usual, the involution is induced by $\ol{t}:=t^{-1}$. Extend this to an involution on $Q=\Q(t)$ and on $Q/\Lambda$. Given a $\L$-module $M$, denote the involuted $\L$-module by $\ol{M}$. This means that $\ol{M}$ as an abelian group is the same as $M$, but given $p\in \L$, the multiplication by $p$ on $\ol{M}$ is defined as left multiplication by $\ol{p}$ on $M$.
\item In an attempt to avoid cluttering the paper and diagrams with more notation than is already necessary, we adopt the convention that for a morphism $f\colon C\to D$ and a covariant functor $F$, we denote the induced morphism $f_* \colon F(C)\to F(D)$ by $f$ as well. On the other hand, if $F$ is contravariant, then we write $f^*\colon F(D)\to F(C)$.
\item When chain, cochain, and homology groups appear without coefficients shown, $\Z$-coefficients are implicit.
\en

\subsection{Twisted homology and cohomology groups}\label{section:twisted-homology}
Let $X$ be a $3$-manifold together with a primitive class $\phi \in H^1(X;\Z)$.
We write $\pi=\pi_1(X)$.
Denote  the infinite cyclic cover of $X$ corresponding to $\phi$ by $p\colon \wti{X}\to X$.
Let $Z\subset Y\subset X$ be two subspaces, and write $\wti{Z}=p^{-1}(Z)$ and $\wti{Y}=p^{-1}(Y)$.
The group $\ll t\rr$ acts on $C_*(\wti{X})$ via deck transformations.
We will view $C_*(\wti{Y},\wti{Z})$ as a chain complex of free $\Lambda=\zt$-modules.

Let $M$ be a $\Lambda$-module.
Write
\[ \ba{rcl} C_*(Y,Z;M)&=&M\otimes_{\L} C_*(\wti{Y},\wti{Z}),\\
C^*(Y,Z;M)&=&\hom_{\L}\Big(\ol{C_*\big(\wti{Y},\wti{Z}\big)},M\Big).\ea\]
These are chain complexes of  $\zt$-modules.
Denote the corresponding homology and cohomology modules by $H_*(Y,Z;M)$ and $H^*(Y,Z;M)$.
We adopt the following conventions.

\bn
\item As usual we drop the $Z$ from the notation when $Z=\emptyset$.
\item With $M=\L$ we will use the canonical isomorphism $\Lambda\otimes_{\Lambda} C_*(\wti{Y},\wti{Z})\xrightarrow{\cong} C_*(\wti{Y},\wti{Z})$ to identify the chain complexes and the corresponding homology modules.
\item With  $M=\L$, we obtain $H_*(X;\L)$ and we refer to the  first homology $H_1(X;\L)$ as  the \emph{Alexander module} of~$X$.  We assume that $H_1(X;\L)$ is a $\L$-torsion module, so that $H_1(X;Q)=0$.
\item If  $Y \subset X$ is connected with $\phi|_Y=0$, pick a lift $Y_0$ of $Y$ to $\wti{X}$ and denote  the corresponding lift of $Z$ to $\wti{X}$ by $Z_0\subset Y_0$. The inclusion induced map
\[ \eta\colon  M\otimes_{\Z} C_*(Y_0,Z_0)\to M\otimes_{\L} C_*(\wti{Y},\wti{Z})=C_*(Y,Z;M)\]
is an isomorphism of $\Lambda$-chain complexes. Similarly, the map
\[ \hspace{1cm}\ba{rcl} \xi \colon M\otimes_{\Z}\hom_{\Z}(C_*(Y_0,Z_0),\Z)&\to & \hom_{\L}\Big(\ol{\L\otimes_{\Z} C_*(Y_0,Z_0)},M\Big)=\hom_{\L}\Big(\ol{C_*(\wti{Y},\wti{Z})},M\Big)\\
p\otimes f&\mapsto & \big(g\otimes \sigma\mapsto f(\sigma)\ol{g}\cdot p\big)\ea\]
is an isomorphism  of $\Lambda$-cochain complexes.
In the sequel we will use these maps to  make the identifications
\[ \ba{rcl} C_*(Y,Z;M)&=&M\otimes_{\Z} C_*(Y_0,Z_0),\\
C^*(Y,Z;M)&=&M\otimes_{\Z}\hom_{\Z}(C_*(Y_0,Z_0),\Z)\ea\]
and
\[ \ba{rcl} H_*(Y,Z;M)&=&M\otimes_{\Z} H_*(Y_0,Z_0),\\
H^*(Y,Z;M)&=&M\otimes_{\Z} \hom_{\Z}(H_*(Y_0,Z_0),\Z).\ea\]
If we want to specify the maps then we refer to them by $\eta$.
\en

\subsection{The evaluation map and the Bockstein map}
\label{section:ev-bockstein}
For $M=\Lambda, Q$ or $M=Q/\Lambda$ the map
\[ \ba{rcl}\kappa\colon \Hom_{\Lambda}\Big(\ol{C_*(\wti{Y},\wti{Z})},M\Big)&\to &
\ol{\op{Hom}_{\L}\big( C_*(\wti{Y},\wti{Z}),M\big)}\\
f&\mapsto & \big(\sigma \mapsto \ol{f(\sigma)}\big)
\ea
\]
gives rise to a well-defined isomorphism of $\Lambda$-modules
\[ \kappa\colon H^i(Y,Z;M)\to H_i\big(\ol{\op{Hom}_{\L}(C_*(Y,Z;\L),M)}\big),
\]
 We also consider the evaluation map
\[ \ba{rcl}
\ev\colon H_i\big(\ol{\op{Hom}_{\L}(C_*(Y,Z;\L),M)}\big)&\to &
\ol{\hom_{\Lambda}\big(H_i(C_*(Y,Z;\Lambda)),M\big)}.
\ea\]
The composition $\ev \circ \kappa$ is sometimes referred to in the literature as the Kronecker evaluation map.

Let $C_*$ be a chain complex  of free  $\Lambda$-modules.
Then there exists a short exact sequence
\[0\to   \Hom_{\L}(C_*,\L)
\to \Hom_{\L}(C_*,Q)\to \Hom_{\L}(C_*,Q/\Lambda)\to 0\]
 of $\Lambda$-modules.
This short exact sequence gives rise to a long exact sequence
\[\ba{rcl} \dots & \to&   H_i(\Hom_{\L}(C_*,\L))
\to H_i(\Hom_{\L}(C_*,Q))\to H_i(\Hom_{\L}(C_*,Q/\Lambda))\to \\
 & \xrightarrow{\op{BS}}&   H_{i+1}(\Hom_{\L}(C_*,\L))
\to \dots\ea
\]
The coboundary map $\op{BS}$ in this long exact sequence is called the Bockstein map.
For example, taking $C_*$ to be the involuted chain complex $\ol{C_*(\wt{X})}$
we obtain a Bockstein map $\op{BS} \colon H^1(X;Q/\Lambda)\to H^2(X;\Lambda)$.

\subsection{Definition of the Blanchfield pairing}
In the following let $X$ be a  $3$-manifold together with a primitive class $\phi \in H^1(X;\Z)$, and suppose that $H_1(X;Q)=0$.
We write $\pi=\pi_1(X)$.

 We denote the composition of the following maps by $\Psi$.
\[ \xymatrix@R0.6cm{H_1(X;\Lambda)\ar@/_2pc/[rddrr]_\Psi\ar[r]& H_1(X,\partial X;\Lambda)\ar[r]^-{\PD}& H^2(X;\Lambda)\ar[r]^-{\BS^{-1}}_-\cong& H^1(X;Q/\Lambda)\ar[d]^-{\kappa}\\
&&& H_1\big(\ol{\hom_{\Lambda}(C_*(X;\Lambda),Q/\Lambda)}\big)\ar[d]^{\ev}\\
&&& \ol{\hom_{\Lambda}(H_1(X;\Lambda),Q/\Lambda)}}.\]
Here we consider the following maps:
\bn
\item the first map is the usual map from absolute to relative homology;
\item PD denotes (the inverse of) Poincar\'e duality, i.e.\ the inverse of capping with the fundamental class of $X$;
\item the maps $\kappa$ and $\ev$ are defined in Section~\ref{section:ev-bockstein}.
\en
We obtain a pairing
\[ \ba{rcl} \op{Bl}\colon H_1(X;\Lambda)\to H_1(X;\Lambda)&\to& Q/\Lambda\\
(a,b)&\mapsto &\Psi(b)(a)\ea\]
which is referred to as the \emph{Blanchfield pairing of $(X,\phi)$}.
It follows immediately from the definitions of the maps that this pairing is sesquilinear over $\Lambda$, in the sense that $\op{Bl}(pa,qb)=p\op{Bl}(a,b)\ol{q}$ for any $a,b \in H_1(X;\Lambda)$ and $p,q\in \Lambda$.  As we already mentioned in Corollary~\ref{cor:hermitian-nonsingular}, the Blanchfield pairing is also hermitian and nonsingular.

\section{The Blanchfield pairing from a dual surface}\label{section:main-technical-theorem}

\subsection{Relating the dual surface to the Blanchfield pairing}
\label{section:31}
 Let $X$ be a  $3$-manifold  together with a class $\phi \in H^1(X;\Z)$, and suppose that $H_1(X;Q)=0$.
Before we state our main result we introduce some more notation and conventions.
Let $F$ be a connected, properly embedded, oriented surface dual  to $\phi$.
\bn
\item  Pick a tubular neighbourhood $F\times I=F\times [-1,1]$ in $X$ and  write $Y=X\sm (F\times (-1,1))$. We take the slight liberty of writing $F\times \pm 1$ instead of $F\times \{\pm 1\}$.
\item Pick a lift $F_0$ of $F$ to the infinite cyclic cover $\wti{X}$ of $X$.
Then $F_0\times [-1,1]$ is also a lift of $F\times [-1,1]$.  Denote  the  unique lift of $Y$ to $\wti{X}$ with the property that it agrees with $F_0\times [-1,1]$ on $F\times 1$ by $Y_0$. Note that with this convention $F_0\times -1$ lies in $t^{-1}Y_0$.
\item  Denote  the inclusion map $F_0\to \wti{X}$ by $\iota$.
\item Denote the composition of the maps $F\to F\times \pm 1\to Y$ by $\iota_\pm$.
\item By a slight abuse of notation we make the identifications $F=F_0$ and $Y=Y_0$.
 Furthermore, by a slight abuse of notation we do not distinguish in our notation between the inclusion maps and the induced maps on homology and cohomology.
\item Denote  the intersection pairing on $H_1(F;\Z)$ by $\underset{F}{\cdot}$. We extend this to a hermitian pairing
\[ \ba{rcl} \Lambda \otimes H_1(F;\Z)\times \Lambda \otimes H_1(F;\Z)&\to & \Lambda\\
(p \otimes v,q \otimes w)&\mapsto & p\big(v\underset{F}{\cdot} w\big)\ol{q}\ea\]
that we also denote by $\underset{F}{\cdot}$. Similarly we extend it to a hermitian pairing on $Q\otimes H_1(F;\Z)$.
\en

Since the Mayer--Vietoris sequence
\[ H_2(X;\Lambda)\to \L \otimes H_1(F)\xrightarrow{\iota_+-t^{-1}\iota_-}\L \otimes H_1(Y)\to H_1(X;\Lambda)\]
is exact and since $H_*(X;Q)=0$ we have that
the map $\iota_+-t^{-1}\iota_-\colon Q \otimes H_1(F)\to Q \otimes H_1(Y)$ is an isomorphism. (In the special case that $X=X_K$ is a knot exterior, see for example \cite{Le77} and the proof of Proposition~\ref{thm:blanchfield-k} for details.)
Now we can state our main technical theorem.

\begin{theorem}\label{thm:computebl}
Let $X$ be a  $3$-manifold together with a class $\phi \in H^1(X;\Z)$, and suppose that $H_1(X;Q)=0$.  
Then for any $v,w\in \L \otimes  H_1(F)$ we have
\[ \op{Bl}(\iota v,\iota w)=-(\iota_+-t^{-1}\iota_-)^{-1}(\iota_+(v))\underset{F}{\cdot} w.\]
\end{theorem}

The proof of Theorem~\ref{thm:computebl} will require the remainder of this section.
In Section~\ref{proof-of-mainthm} we will deduce Theorem~\ref{mainthm} from Theorem~\ref{thm:computebl}.

\subsection{Proof of Theorem~\ref{thm:computebl}}\label{section:proof-of-theorem}
Throughout this section 
let $X$ be a  $3$-manifold  together with a class $\phi \in H^1(X;\Z)$ for which $H_1(X;Q)=0$,
and let $F$ be a dual surface corresponding to~$\phi$. We write $Y=X\sm F\times (-1,1)$.


In order to keep track of the various inclusion maps we will denote
any inclusion map $A\to B$ by $\iota(A)$. We consider the following commutative diagram of $\Lambda$-homomorphisms:

\[ \xymatrix@C0.5cm@R0.7cm{H_1(X;\Lambda)\ar[r]& H_1(X,\partial X;\Lambda)\ar[r]^-{\PD}& H^2(X;\Lambda)\ar[d]^\kappa\ar[rr]^-{\BS^{-1}}_-\cong&& H^1(X; Q/\Lambda)\ar[d]^-{\kappa}\\
\Lambda \otimes  H_1(F)\ar@/_2pc/[rr]_{=:\Upsilon}\ar[u]^{\iota(F)\circ \eta}&&
H_2\big(\ol{\hom_{\Lambda}(C_*(X;\Lambda),\Lambda)}\big)\ar@/_2pc/[rrdd]_-{=:\Omega}\ar[rr]^-{\BS^{-1}}_-{\cong}&& H_1\big(\ol{\hom_{\L}(C_*(X;\Lambda),Q/\Lambda)}\big)\ar[d]^{\ev}\\
&&&& \ol{\hom_{\Lambda}(H_1(X;\Lambda),Q/\Lambda)}\ar[d]^{(\iota(F) \circ \eta)^*}\\
&&&& \ol{\hom_\L(\Lambda \otimes H_1(F),Q/\Lambda).}}\]

Our first goal is to understand the maps $\Upsilon$ and $\Omega$.
We start out with $\Upsilon$.

Pick a CW-structure for $F$ and equip $F\times [-1,1]$ with the corresponding product CW-structure. We extend this to a CW-structure over $Y$ and thus over $X$.
We introduce the following notation and conventions.
\bn
\item Given a chain complex $C_*$ we denote  the cycles in $C_k$ by $Z_k(C_*):=\ker(\partial\colon C_k\to C_{k-1})$. We denote the projection map $Z_k(C_*)\to H_k(C_*)$ by~$p$.
\item Pick a splitting $b\colon C_1(F)\to Z_1(F)$ of the inclusion map $Z_1(F)\to C_1(F)$. We may do so because the short exact sequence $0 \to Z_1(F) \to C_1(F) \to B_0(F)\to 0$, where $B_0(F) = \im(C_1(F) \to C_0(F))$, splits, since $B_0(F)$ is a submodule of the free $\Z$-module $C_0(F)$ and hence is also free.
\item Identify $C_k(F\times I,F\times \pm 1)$ with the free $\Z$-module generated by the open product $k$-cells.
\item Let $\times I\colon C_*(F)\to C_{*+1}(F\times I,F\times \pm 1)$ be the chain isomorphism that is induced by mapping each $i$-dimensional cell in $F$ to the corresponding open $(i+1)$-dimensional product cell.
\item Let $\iota(F\times I)\colon C_k(\wti{F\times I},\wti{F\times \pm 1})\to C_k(\wti{X})$ be the inclusion map of $\Lambda$-modules.  That is, $\iota(F \times I)$ is the map which sends the open product cell in $\wti{F\times I}$ to the same cell in $\wti{X}$.
By a slight abuse of notation we also denote the induced map $C_k(F\times I,F\times \pm 1;\Lambda)\to C_k(X;\Lambda)$ of $\Lambda$-modules by $\iota(F\times I)$.
\item Let
\[ c\colon {C_2(\wti{X})}\to {C_2(\wti{F\times I},\wti{F\times \pm 1})}\]
 be the chain map of $\Lambda$-modules which is  the identity on the open product cells and which is the zero map on all other cells.
  By a slight abuse of notation we denote the induced map
 \[ c\colon {C_2(X;\Lambda)}\to C_2(F\times I,F\times \pm 1;\Lambda)\]
 also by $c$.
 Note that $c$ is a splitting of $\iota(F\times I)$.
\item Given a map $f\colon A\to B$ between $\Z$-modules we denote the induced maps  $\Lambda \otimes A\to \Lambda \otimes B$ and $Q \otimes A\to Q\otimes B$ by $f$ as well.
\en
Now let $w\in H_1(F)$.
We consider the following diagram of $\Lambda$-module homomorphisms.  Recall that the map $\eta$, defined at the end of Section~\ref{section:twisted-homology}, can be used for the chains of $F_0$ or $Y_0$.

\[ \xymatrix@R1.0cm@C1.1cm{
&&&&\Lambda\\
{\Lambda} \otimes H_1(F)\ar@/^1.3pc/[urrrr]^{v\mapsto -v\underset{F}{\cdot}w}
&{\Lambda} \otimes Z_1(F)\ar@{^{(}->}[r]\ar@{->>}[l]^-{p} & {\Lambda} \otimes C_1(F)\ar[r]^-{\times I}_-\cong  \ar@/_1pc/[l]_-b&
{\Lambda} \otimes C_2(F\times I,F\times \pm 1)\ar@{^{(}->}[r]_-{\iota(F\times I)\circ \eta}&
{C_2(X;\Lambda)}\ar@/_1.5pc/[l]_-{\eta^{-1}\circ c}\ar[u]^{=:\phi_w}.\\
 }\]
We have the following lemma.

\begin{lemma}\label{claim1}
The homomorphism $\phi_w\colon C_2(X;\Lambda)\to \Lambda$
  defined by the above sequence of homomorphisms represents $\Upsilon(1\otimes w)$.
\end{lemma}

\begin{proof}
Let $\beta\colon \Lambda\to \Lambda$ be the map given by $\beta(p)=\ol{p}$.
Now consider  the following diagram of maps of $\L$-modules.
 \[
 \xymatrix@C0.3cm@R0.7cm{
 &&&\Lambda\hspace{-0.1cm}\otimes\hspace{-0.1cm} \hom_\Z(H_1(F);\Z)\ar[r]^-{\beta \otimes \Id}& \ol{\Lambda}\hspace{-0.1cm}\otimes\hspace{-0.1cm}\hom_{\Z}(H_1(F),\Z)
 \ar@/^2pc/[d]^{(p\circ b)^*}
 \\
 \Lambda \hspace{-0.1cm}\otimes\hspace{-0.1cm} H_1(F)\ar[r]\ar[d]_\cong^{\iota}\ar@/^2pc/[urrr]^{w\mapsto (v\mapsto -v\underset{F}{\cdot}w)}&\Lambda\hspace{-0.1cm}\otimes\hspace{-0.1cm} H_1(F,\partial F)\ar[rr]_{\cong}^-{\Id\otimes - \PD}\ar[d]_\cong^{\iota}&&\Lambda \hspace{-0.1cm}\otimes\hspace{-0.1cm} H^1(F) \ar[r]_-\cong^-{\beta\otimes \Id}\ar[u]_{\ev}^\cong
 & \ol{\Lambda}\hspace{-0.1cm}\otimes\hspace{-0.1cm} H_1(\hom_{\Z}(C_*(F),\Z))\ar[u]^{\ev}_\cong
 \\
 \Lambda \hspace{-0.1cm}\otimes\hspace{-0.1cm} H_1(F\hspace{-0.09cm}\times \hspace{-0.09cm}I)\ar[r]\ar[d]_\cong^\eta&\Lambda \hspace{-0.1cm}\otimes\hspace{-0.1cm} H_1(F\hspace{-0.09cm}\times \hspace{-0.09cm}I,\partial F\hspace{-0.09cm}\times \hspace{-0.09cm}I)\ar[rr]_{\cong}^-{\hspace{-0.2cm}\Id\hspace{-0.05cm}\otimes\hspace{-0.05cm} \PD}\ar[d]_\cong^\eta&&\Lambda \hspace{-0.1cm}\otimes\hspace{-0.1cm} H^2(F\hspace{-0.09cm}\times \hspace{-0.09cm}I,F_\pm)\ar[u]^\cong_{(\times I)^*} \ar[d]_{\cong}^{\xi}
 \ar[r]^-{\hspace{-0.15cm}\beta\otimes \Id}_-{\hspace{-0.15cm}\cong}& \ol{\Lambda}\hspace{-0.1cm}\otimes\hspace{-0.1cm} H_2(\hom_{\Z}(C_*(F\hspace{-0.09cm}\times \hspace{-0.09cm} I,F_\pm),\Z)) \ar[d]_{\cong}^{\xi} \ar[u]^{\cong}_{(\times I)^{*}}
 \\
 H_1(F\hspace{-0.09cm}\times \hspace{-0.09cm}I;\Lambda)\ar[d]^{\iota(F\hspace{-0.02cm}\times \hspace{-0.02cm}I)}\ar[r] &H_1(F\hspace{-0.09cm}\times \hspace{-0.09cm} I,\partial F\hspace{-0.09cm}\times \hspace{-0.09cm} I;\Lambda)\ar[d]^{\iota(F\hspace{-0.02cm}\times \hspace{-0.02cm}I)}\ar[rr]_-\cong^-{\PD}&&H^2(F\hspace{-0.09cm}\times \hspace{-0.09cm} I,F_\pm;\Lambda)\ar[d]^{c^*} \ar[r]^-\kappa_-\cong&
 H_2\big(\ol{\hom_{\Lambda}(C_*(F\hspace{-0.09cm}\times \hspace{-0.09cm} I,F_\pm;\Lambda),\Lambda)}\big) \ar[d]^-{c^*}\\
 H_1(X;\Lambda)\ar[r]& H_1(X,\partial X;\Lambda)\ar[rr]_-\cong^-{\PD} && H^2(X;\Lambda)
 \ar[r]^-{\kappa}_-\cong& H_2\big(\ol{\hom_{\Lambda}(C_*(X;\Lambda),\Lambda)}\big).}
 \]
Of course we could have $\partial F = \emptyset=\partial X$.
We assert that the above diagram commutes.
\bn
\item The bottom middle square involving Poincar\'{e} duality commutes by Bredon \cite[Corollary~VI.8.4]{Br93}.
\item
We argue that the top middle square involving Poincar\'{e} duality commutes, and in particular that the given sign change on $-\op{PD}$ is correct. We start out with an observation. Let  $X$ be an oriented $n$-manifold with a decomposition of the boundary $X=A\cup B$ with $\partial A=A\cap B=\partial B$. Let $a\in H_k(X,A)$ and $b\in H_{n-k}(X,B)$.
We write $\alpha=\op{PD}(a)\in H^{n-k}(X,B)$ and $\beta=\op{PD}(b)\in H^k(X,A)$. Put differently, $a=\alpha\cap [X,\partial X]$ and $b=\beta\cap [X,\partial X]$.  By \cite[p.~367]{Br93}, by the graded commutativity of the cup product and by \cite[Section~VI.5.3]{Br93} we have
\[\hspace{0.9cm} \ba{rcl}a\underset{X}{\cdot} b&=&\op{PD}(\beta\cup \alpha)=(\beta\cup \alpha)\cap [X,\partial X]\\
&=&\ll \beta\cup \alpha,[X,\partial X]\rr =(-1)^{k(n-k)} \ll \alpha\cup \beta,[X,\partial X]\rr=
(-1)^{k(n-k)} \ll \alpha,\beta\cap [X,\partial X]\rr\\
&=&(-1)^{k(n-k)}\ll \alpha,b\rr.\ea \]
Here $\cdot_X$ denotes the intersection pairing on $X$ and $\ll-,-,\rr$ is the Kronecker evaluation map $H^i(X,Y)\times H_i(X,Y)\to \Z$.
Now  consider the  diagram
\[\hspace{0.9cm}
\xymatrix@C+1.8cm{ H_1(F,\partial F)\ar[d]^\iota\ar@/^2pc/[rr]^{a\mapsto (b\mapsto a\underset{F}{\cdot} b)}& H^1(F)\ar[l]^-\cong_-{-\alpha \cap [F,\partial F]\mapsfrom \alpha}
\ar[r]_-{\cong}^-{\alpha \mapsto (b\mapsto \ll \alpha,b\rr)}&\hom(H_1(F),\Z)\\
 H_1(F\times I,\partial F\times I)\ar@/_2pc/[rr]_{a\mapsto (b\mapsto a\underset{F\times I}{\cdot} b)}& H^2(F\times I,F_\pm)\ar[l]^-\cong_-{ \alpha \cap [F\times I,\partial (F \times I)]\mapsfrom \alpha}
\ar[r]_-{\cong}^-{\alpha \mapsto (b\mapsto \ll \alpha,b\rr)}\ar[u]^\cong_{(\times I)^*}&\hom(H_2(F\times I,F_\pm),\Z)\ar[u]^\cong_{(\times I)^*}.}\]
It follows from the above discussion that the two curved maps are the composition of the corresponding two horizontal maps. It is straightforward to see, using the geometric interpretation of the intersection pairing, that outside quadrilateral commutes. Evidently the right hand rectangle commutes. Since all maps are isomorphisms it follows that, as claimed, the left hand rectangle commutes.
\item It is straightforward to see that all other squares commute.
\item The  part of the diagram on top involving the big curved map commutes by the definition of the tensored up intersection pairing on $\Lambda\otimes H_1(F)$.
\en
Now the lemma follows from the fact that the diagram commutes and from the definitions. To see this, first observe that travelling top left to bottom left, then to bottom right, is the definition of $\Upsilon$.  On the other hand passing along $w\mapsto -\underset{F}{\cdot}w$, then down to the bottom right, gives the homomorphism $\phi_w$.
\end{proof}

Now we turn to the map $\Omega \colon H_2\big(\ol{\hom_{\Lambda}(C_*(X;\Lambda),\Lambda)}\big) \to \ol{\hom_\L(\Lambda \otimes H_1(F),Q/\Lambda)}$.
\bn
\item Since $H_1(F)$ is torsion-free we can pick a splitting $a\colon H_1(F)\to Z_1(F)$ of the projection map $Z_1(F)\to H_1(F)$.
\item
Since $H_1(X;Q)=0$ and since $Q$ is a field we can and will pick a splitting
$d\colon Z_1(X;Q)\to C_2(X;Q)$ of the boundary map $\partial \colon C_2(X;Q)\to Z_1(X;Q)$.
\item Given $\psi \in \hom_{\Lambda}(C_2(X;\Lambda),\Lambda)$ we denote the corresponding homomorphism $\hom_{Q}(C_2(X;Q),Q)$ by $\psi^Q$.
\en

\begin{lemma}\label{claim2}
The  homomorphism
\[ \ba{rcl} H_2\big(\ol{\hom_{\Lambda}(C_*(X;\Lambda),\Lambda)}\big) &\to & \ol{\hom_{\Lambda}(\Lambda \otimes  H_1(F),Q/\Lambda)} \\
\psi&\mapsto & \left(\ba{rcl} \Lambda \otimes  H_1(F)&\to &Q/\Lambda\\
p\otimes w &\mapsto &p\cdot \psi^Q((d\circ \iota_+\circ \eta\circ a)(w))\ea\right)\ea\]
of $\Lambda$-modules
is precisely the homomorphism $\Omega$.
\end{lemma}

\begin{proof}
Consider the following diagram of $\Lambda$-modules:
\[ \xymatrix@R0.85cm@C0.4cm{
&\ol{{\op{Hom}}_{\Lambda}\hspace{-0.05cm}(C_1(X;\Lambda),Q)}\ar[dd]^{\partial^*}\ar[r]\ar[dr]&\ol{{\op{Hom}}_{\Lambda}\hspace{-0.05cm}(C_1(X;\Lambda),Q/\Lambda)}\ar[dr]\\
&&\ol{{\op{Hom}}_{\Lambda}\hspace{-0.05cm}(Z_1(X;\Lambda),Q)}\ar[r]\ar[d]^=&\ol{{\op{Hom}}_{\Lambda}\hspace{-0.05cm}(Z_1(X;\Lambda),Q/\Lambda)}\ar[d]^{\iota_+^*}\\
\ol{{\op{Hom}}_{\Lambda}\hspace{-0.05cm}(C_2(X;\Lambda),\Lambda)}\ar[r]\ar[dr]_{\psi\mapsto \psi^Q}&\ol{{\op{Hom}}_{\Lambda}\hspace{-0.05cm}(C_2(X;\Lambda),Q)}\ar[d]^=&\ol{{\op{Hom}}_{Q}(Z_1(X;Q),Q)}\ar@/_0pc/[dl]_{\partial^*}& \ol{{\op{Hom}}_{\Lambda}\hspace{-0.05cm}(Z_1(F;\Lambda),Q/\Lambda)}\ar[d]^{\eta^*}_\cong\\
&\ol{{\op{Hom}}_{Q}(C_2(X;Q),Q)}\ar@/_1.5pc/[ur]_{d^*}&&\ol{{\op{Hom}}_{\Lambda}\hspace{-0.05cm}(\Lambda \hspace{-0.09cm}\otimes \hspace{-0.09cm}Z_1(F),Q/\Lambda)}\ar[d]^{a^*}\\
&&&\ol{{\op{Hom}}_{\Lambda}\hspace{-0.05cm}(\Lambda \hspace{-0.09cm}\otimes\hspace{-0.09cm} H_1(F),Q/\Lambda)}.
}\]
It is straightforward to verify that the diagram commutes. The map $\Omega$ is defined by the uppermost route.  The second step in this route involves choosing a lift to $\ol{\hom_{\Lambda}(C_1(X;\Lambda),Q)}$ of an element in $\ol{\hom_{\Lambda}(C_2(X;\Lambda),Q)}$.
On the other hand  $d\circ\iota_+\circ \eta\circ a$ is given by the sequence of maps in the lowermost route of the diagram.  Since the diagram commutes we get the same result.
This concludes the proof of the lemma.
\end{proof}

We are now ready to give the proof of Theorem~\ref{thm:computebl}.



\begin{proof}[Proof of Theorem~\ref{thm:computebl}]

Let $w\in \Lambda\otimes H_1(F)$. We define $\phi_w\colon C_2(X;\Lambda)\to \Lambda$ as in Lemma~\ref{claim1}.
Consider the following diagram of $\Lambda$-module homomorphisms
\[ \xymatrix@R0.8cm@C1.0cm{
&&&&Q/\Lambda\\
&&&&Q\ar@{->>}[u]\\
{Q} \otimes H_1(F)\ar@/^1pc/[rrrru]^{v\mapsto -v\underset{F}{\cdot}w}
\ar@/_0pc/[d]_{\iota_+-t^{-1}\iota_-}^\cong&{Q} \otimes Z_1(F)\ar@{^{(}->}[r]\ar@{->>}[l]^p \ar[d]^{\iota_+-t^{-1}\iota_-}& {Q}  \otimes C_1(F)\ar[r]^-{\times I}_-\cong  \ar@/_1pc/[l]_-b&
{Q}  \otimes C_2(F\times I,F\times \pm 1)\ar@{^{(}->}[r]_-{\iota(F\times I)\circ \eta}&
C_2(X;{Q} )\ar@{->>}[d]_\partial\ar@/_1pc/[l]_-{\eta^{-1}\circ c}\ar[u]^{\phi_w^Q}\\
 {Q}  \otimes H_1(Y)& {Q}  \otimes Z_1(Y)\ar@{->>}[l]^p\ar@{^{(}->}[rrr]^{\iota(Y)\circ \eta} &&&Z_1(X;{Q})\ar@/_1pc/[u]_d\\
{Q}  \otimes H_1(F) \ar[u]^{\iota_+} \ar@{_{(}->}@/_1pc/[r]_{a}& \ar@{->>}[l]_p {Q} \otimes Z_1(F)\ar@{^{(}->}[u]^{\iota_+}&&&\\
{\Lambda}  \otimes H_1(F) \ar@{^{(}->}[u] \ar@{_{(}->}@/_1pc/[r]_{a}& \ar@{->>}[l]_p {\Lambda} \otimes Z_1(F)\ar@{^{(}->}[u]\ar@{^{(}->}[uurrr]_{\iota_+\circ \eta}
 }\]
Here the top left vertical map $\iota_+-t^{-1}\iota_-$ arises from the conventions of Section~\ref{section:31} that $F_0\times 1$ lies in $Y_0$ and  $F_0\times -1$ lies in $t^{-1}Y_0$.
It follows easily from the definitions that the homomorphism $C_2(X;Q)\to Q$ on the top right is indeed the homomorphism $\phi_w^Q$ given by tensoring up $\phi_w\colon C_2(X;\Lambda)\to \Lambda$ with~$Q$.
Here the left squares and the right rectangle with the horizontal arrows are easily seen to commute, with any choice of maps indicated in the squares and rectangles. The top triangle commutes by definition of $\phi_w^Q$. It is straightforward to see that the bottom  triangle commutes.

The theorem will follow from  the two lemmas above and the commutativity of the diagram, as we now explain.
Let $v\in \Lambda \otimes H_1(F)$. We write
\[ z=(b\circ (\times I)^{-1}\circ\eta^{-1}\circ c\circ d\circ \iota_+ \circ \eta \circ a)(v)\in Q \otimes Z_1(F).\]
That is, $z\in Q \otimes Z_1(F)$ is the element obtained from $w$ via the long route on the right. It follows from the previous two lemmas that:
\[ \ba{rcl} \op{Bl}(\iota v,\iota w)&=&\Omega(\Upsilon(w))(v)\\
&=& \Omega(\phi_w)(v) \\
&=&(\phi_w^Q\circ d\circ \iota_+\circ\eta\circ a)(v)\\
&=&p(z)\underset{F}{\cdot} w\ea \]
More precisely, the first equality follows from the diagram at the beginning of Section~\ref{section:proof-of-theorem}.  The second equality follows from Lemma~\ref{claim1}, and the third equality follows from Lemma~\ref{claim2}.  The fourth and final equality follows from the definition of $z$ and the commutativity of the top triangle.  This top triangle is the triangle that appeared above Lemma~\ref{claim1}, tensored with $Q$ over $\L$, and it commutes by definition of the right hand map.

 We make the following observations to finish the proof.
\bn
\item Since the big rectangle  commutes and since $Q \otimes Z_1(Y)\to Z_1(X;Q)$ is a monomorphism we see that $(\iota_+-t^{-1}\iota_-)(z)=(\iota_+\circ a)(v)$.
\item Since the bottom two squares and the bottom triangle commute we see that the image of $p((\iota_+-t^{-1}\iota_-)(z))=\iota_+(v)$.
\item Since the top left square commutes and since $\iota_+-t^{-1}\iota_-\colon Q \otimes H_1(F)\to Q \otimes H_1(Y)$ is an isomorphism, we see that $p(z)=(\iota_+-t^{-1}\iota_-)^{-1}(\iota_+)(v)$.
\en
Combining these observations with the equation $\op{Bl}(\iota v,\iota w) = p(z)\underset{F}{\cdot} w$ from above, we obtain that
\[ \op{Bl}(\iota v,\iota w)=-p(z)\underset{F}{\cdot} w=-(\iota_+-t^{-1}\iota_-)^{-1}(\iota_+)(v)\underset{F}{\cdot} w.\]
This concludes the proof of Theorem~\ref{thm:computebl}.

\end{proof}

\section{Proof of Theorem~\ref{mainthm}}
\label{proof-of-mainthm}

In this section we will prove Theorem~\ref{thm:blanchfield-k-intro}, which implies Theorem~\ref{mainthm}.
Throughout this section let $K$ be an oriented knot in $S^3$, let $X_K := S^3 \sm \nu K$ be the knot exterior, and we will apply Theorem~\ref{thm:computebl} with $\phi \in H^1(X_K;\Z) = \Hom(\pi_1(X_K),\Z)$ the epimorphism corresponding to the orientation.  Denote the Blanchfield pairing of~$K$ by~$\Bl_K$.

Given a Seifert surface $F$ for $K$ and a basis
$\{d_1,\dots,d_{2g}\}$  for $H_1(F;\Z)$ we write
$a_{ij}=\op{lk}(d_i,d_j^+)$ and we refer to $A=(a_{ij})$ as a Seifert matrix for $K$. Here we follow the convention of  \cite[p.~201]{Ro90}\cite[p~.53]{Li97}, which is the transpose of the convention used in~\cite{Ke75} and \cite{Le77}. For the convenience of the reader we restate Theorem~\ref{thm:blanchfield-k-intro} here.

\begin{theorem}\label{thm:blanchfield-k}
Let  $K\subset S^3$ be an oriented knot.
Let $F$ be a Seifert surface for the knot $K$ and let  $\{d_1,\dots,d_{2g}\}$ be a basis for $H_1(F;\Z)$.
Denote the Seifert matrix of $K$ corresponding to the basis $\{d_1,\dots,d_{2g}\}$  by $A$ and
 write  $Y:= X \sm (F \times(-1,1))$.  Let $\{e_1,\dots,e_{2g}\}$ be the basis of $H_1(Y;\Z)$ that is  Alexander dual to the basis $\{d_1,\dots,d_{2g}\}$, that is $\op{lk}(d_i,e_j)=\delta_{ij}$.
Then the homomorphism
\[ \ba{rcl} \Phi\colon \Lambda^{2g}&\to& H_1(X_K;\Lambda)\\
  (p_1,\dots,p_{2g})&\mapsto & \sum\limits_{i=1}^{2g}p_ie_i \ea\]
induces an isomorphism
\[  \Phi\colon \Lambda^{2g}/(A-t^{-1}A^T)\Lambda^{2g}\xrightarrow{\cong} H_1(X_K;\Lambda)\]
which in turn gives rise to a commutative diagram
\[ \xymatrix@C4cm{\Lambda^{2g}/(A-t^{-1}A^T)\Lambda^{2g}\times \Lambda^{2g}/(A-t^{-1}A^T)\Lambda^{2g}\ar[r]^-{(v,w)\mapsto v^T(t-1)(A - tA^T)^{-1}\ol{w}}\ar[d]^{\Phi\times \Phi}&Q/\Lambda\ar[d]^=\\
H_1(X_K;\Lambda)\times H_1(X_K;\Lambda)\ar[r]^-{(v,w)\mapsto \Bl_K(v,w)}&Q/\Lambda.}\]
\end{theorem}

We remark that the presentation matrix used for the Alexander module is slightly different from the matrix given in the introduction, but the two matrices only differ by multiplication by a unit $t^{\pm 1}$.

\begin{proof}
By definition of the Seifert form, with respect to the bases $\{d_1,\dots,d_{2g}\}$ and $\{e_1,\dots,e_{2g}\}$,
the two inclusion induced maps $\iota_-\colon H_1(F;\Z)\to H_1(F\times -1;\Z) \to H_1(Y;\Z)$ and $\iota_+ \colon H_1(F;\Z)\to H_1(F\times 1;\Z) \to H_1(Y;\Z)$ are represented  by $A^T$ and $A$ respectively.
Let
\[ \ba{rcl} \Theta\colon \Lambda^{2g}&\to& \Lambda \otimes H_1(F;\Z)\\
  (p_1,\dots,p_{2g})&\mapsto & \sum\limits_{i=1}^{2g}p_id_i \ea\mbox{ and }
  \ba{rcl} \wt{\Phi}\colon \Lambda^{2g}&\to& \Lambda \otimes H_1(Y;\Z)\\
    (p_1,\dots,p_{2g})&\mapsto & \sum\limits_{i=1}^{2g}p_ie_i \ea\]
be the isomorphisms induced by the choices of bases.
Now we note that the Mayer--Vietoris sequence and the above maps yield the following commutative diagram of short exact sequences:
\[ \xymatrix@C0.8cm{0\ar[r]& \Lambda^{2g}\ar[d]^\Theta\ar[rr]^{A-t^{-1}A^T}&& \Lambda^{2g}\ar[d]^{\wt{\Phi}}\ar[r] & \Lambda^{2g}/(A^T - tA)\Lambda^{2g}\ar[d]^{\Phi}\ar[r]&0\\
 0\ar[r]& \Lambda \otimes H_1(F;\Z)\ar[rr]^{\iota_+-t^{-1}\iota_-}&& \Lambda \otimes H_1(Y;\Z)\ar[r]& H_1(X_K;\Lambda)\ar[r] &0.}\]
This concludes the proof of the (well known) first statement of Theorem~\ref{thm:blanchfield-k}. Before we continue with the discussion of the Blanchfield pairing we need to prove the following lemma.

\begin{lemma}\label{lemma:iota-+-surjective}
The map
\[ \iota_+\colon \Lambda \otimes H_1(F;\Z) \to
\coker\big((\iota_+-t^{-1}\iota_-)\colon \Lambda \otimes H_1(F;\Z)\to
\Lambda \otimes H_1(Y;\Z)\big) \]
is surjective.
\end{lemma}


\begin{proof}
Since the map is linear over $\Lambda$ it suffices to show that any element on the right hand side that is represented by an element in $H_1(Y;\Z)$ lies in the image.
So let $w\in H_1(Y;\Z)$.
First note that $\iota_+-\iota_-\colon H_1(F;\Z)\to H_1(Y;\Z)$ is an isomorphism since it can be represented by the matrix $A-A^T$ which has determinant $\pm 1$. Thus we see that the images of $\iota_+$ and $\iota_-$ generate $H_1(Y;\Z)$. In particular we can find $u,v\in H_1(F;\Z)$ such that $w=\iota_+(u)+\iota_-(v)$. It follows that
\[ w = \iota_+(u)+\iota_-(v)=
\iota_+(u+tv)+(\iota_+-t^{-1}\iota_-)(-tv).\]
This concludes the proof of the lemma.
\end{proof}

Now we turn to the discussion of the Blanchfield pairing. First note that by \cite[Chapter~8]{Ro90} we have the following commutative diagram
\[ \xymatrix@C3cm{ \Lambda^{2g}\times \Lambda^{2g} \ar[r]^-{(v,w)\mapsto v^T(A-A^T)\ol{w}}\ar[d]^{\Theta\times \Theta}&\Lambda\ar[d]^=\\
\Lambda \otimes H_1(F;\Z)\times \Lambda \otimes H_1(F;\Z)\ar[r]^-{(v,w)\mapsto v\cdot_F w}&\Lambda.}\]
Thus by Theorem~\ref{thm:computebl}, given any $v,w\in \Lambda^{2g}$, we have
\be \label{equ:computebl}\Bl_K(\iota_+(\Theta(v)),\iota_+(\Theta(w))) =- \big((A-t^{-1}A^T)^{-1}Av\big)^T  (A-A^T)\ol{w}=-
(Av)^T(A^T-t^{-1}A)^{-1}(A-A^T)\ol{w}.\ee
Now we turn to the proof that the diagram in Theorem~\ref{thm:blanchfield-k} is in fact commutative. So let $\wt{v},\wt{w}\in \Lambda^{2g}$.
First we consider the special case that $\wt{\Phi}(\wt{v})$ and $\wt{\Phi}(\wt{w})$ lie in the image of $\iota_+\colon \Lambda \otimes H_1(F;\Z)\to \Lambda \otimes H_1(Y;\Z)$. This means  there exist $v,w\in \Lambda^{2g}$ such that
$$\iota_+(\Theta(v))=\wt{\Phi}(\wt{v}) \text{ and } \iota_+(\Theta(w))=\wt{\Phi}(\wt{w}).$$
The commutative diagram
\[ \xymatrix{ \Lambda \otimes H_1(F;\Z)\ar[r]^{\iota_+}& \Lambda \otimes H_1(Y;\Z)\\
\Lambda^{2g}\ar[u]^{\Theta}_\cong \ar[r]^{A} & \Lambda^{2g}\ar[u]^{\wt{\Phi}}_\cong}\]
implies that $Av=\wt{v}$ and $Aw=\wt{w}$.
By (\ref{equ:computebl}) we have
\[ \ba{rcl} \Bl_K(\wt{\Phi}(\wt{v}),\wt{\Phi}(\wt{w}))&=&-\Bl_K(\iota_+(\Theta(v)),\iota_+(\Theta(w)))\\[0.05cm]
 &=&-(Av)^T(A^T-t^{-1}A)^{-1}(A-A^T) \ol{w}\\[0.05cm]
  &=&-(Av)^T(A^T-t^{-1}A)^{-1}((-A^T+t^{-1}A)+(A-t^{-1}A)) \ol{w}\\[0.05cm]
   &=&-(Av)^T(A^T-t^{-1}A)^{-1}(A-t^{-1}A) \ol{w}\\[0.05cm]
    &=&-(Av)^T(A^T-t^{-1}A)^{-1}(1-t^{-1}) \ol{Aw}\\[0.05cm]
  &=&-\wt{v}^T(A^T-t^{-1}A)^{-1}(1-t^{-1}) \ol{\wt{w}}\\[0.05cm]
   &=&\wt{v}^T(A-tA^T)^{-1}(t-1) \ol{\wt{w}}\in Q/\Lambda.\ea\]
In the step from the third to the fourth line we used that we are only interested in equality modulo~$\Lambda$.

Now we turn to the general case. So let $\wt{v},\wt{w}\in \Lambda^{2g}$.
By Lemma~\ref{lemma:iota-+-surjective}, there exist $v,w, x, y\in \Lambda^{2g}$ such that
$$\wt{\Phi}(\wt{v}) = \iota_+(\Theta(v)) + (\iota_+ - t^{-1} \iota_-)(\Theta(x)) \text{ and } \wt{\Phi}(\wt{w}) = \iota_+(\Theta(w)) + (\iota_+ - t^{-1} \iota_-)(\Theta(y)).$$
The two commutative diagrams for the price of one
\[ \xymatrix @C+1cm{ \Lambda \otimes H_1(F;\Z)\ar[r]^{\{\iota_+,\iota_+-t^{-1}\iota_-\}}& \Lambda \otimes H_1(Y;\Z)\\
\Lambda^{2g}\ar[u]^{\Theta}_\cong \ar[r]^{\{A,A-t^{-1}A^T\}} & \Lambda^{2g}\ar[u]^{\wt{\Phi}}_\cong}\]
imply that $\wt{v} = Av + (A-t^{-1}A^T)x$ and $\wt{w} =Aw +(A-t^{-1}A^T)y$.
Now observe that
\[
\ba{rcl}
  \wt{v}^T(A - tA^T)^{-1}(t-1)\ol{\wt{w}}
  &=& \big(Av + (A-t^{-1}A^T)x\big)^T (A - tA^T)^{-1}(t-1)\big(\ol{Aw +(A-t^{-1}A^T)y}\big)\\[0.05cm]
  &=& \big((Av)^T + x^T(A^T-t^{-1}A)\big)(A - tA^T)^{-1}(t-1)\big(\ol{Aw} +(A-tA^T)\ol{y}\big)\\[0.05cm]
  &=& \big((Av)^T -t^{-1} x^T(A-tA^T)\big)(A -tA^T)^{-1}(t-1)\big(\ol{Aw} +(A-tA^T)\ol{y}\big)\\[0.05cm]
  &=&(Av)^T(A - tA^T)^{-1}(t-1) \ol{A w},\ea\]
where the last equality holds in $Q/\Lambda$ only.
Combine this computation with the computation above to obtain $\Bl_K(\wt{\Phi}(\wt{v}),\wt{\Phi}(\wt{w})) = \wt{v}^T(A - tA^T)^{-1}(t-1)\ol{\wt{w}}$
as desired.
\end{proof}

\section{Blanchfield pairings of fibred 3-manifolds}\label{section:Bl-form-fibred-manifold}

In this section we apply Theorem~\ref{thm:computebl-intro} to give the proof of Corollary~\ref{cor:fib} from the introduction, that computes the Blanchfield pairing of fibred 3-manifolds.  We restate the corollary here for the convenience of the reader.

\begin{corollary}\label{cor:fib-text}
Let $F$ be a  surface and let $\varphi\colon F\to F$ be a self-homeomorphism.
Denote the mapping torus by $M=M(F,\varphi)$ and denote the canonical epimorphism by $\phi \colon \pi_1(M) \to \Z$.  Pick a basis $c_1,\dots,c_k$ for $H_1(F;\Z)$. With respect to this basis, let $J$ be the matrix representing the intersection pairing on $H_1(F;\Z)$, and let $P$ be the matrix representing the monodromy $\varphi_*\colon H_1(F;\Z)\to H_1(F;\Z)$. Then the Blanchfield pairing of $(M,\phi)$ is isomorphic to
\[ \ba{rcl} \L^k/(tP-\id)\times \L^k/(tP-\id)&\to &Q/\Lambda\\
(v,w)&\mapsto & v^T J(t^{-1}P-\id)^{-1}\ol{w}.\ea\]
\end{corollary}

\begin{proof}
In the following we identify $F$ with $F\times 0$ and we view $F\times [0,\frac{1}{4})\cup (\frac{3}{4},1]$ as the tubular neighbourhood of $F$.
Correspondingly we write $Y=F\times [\frac{1}{4},\frac{3}{4}]\subset M$. We equip $H_1(Y;\Z)$ with the basis given by $d_i=\varphi^{-1}(c_i)\times \frac{1}{2}$, $i=1,\dots,k$. We use these bases to identify $H_1(F;\Z)$ with $\Z^k$ and to identify $H_1(Y;\Z)$ with $\Z^k$.

With respect to these bases, the map $$\iota_+\colon H_1(F;\Z)=H_1(F\times 0;\Z)\to H_1(F\times\tmfrac{1}{4};\Z)\to  H_1(Y;\Z)$$ is represented by $P$  and the map $$\iota_-\colon  H_1(F;\Z)=H_1(F\times 0;\Z)\to H_1(F\times 1;\Z) \to H_1(F\times\tmfrac{3}{4};\Z)\to   H_1(Y;\Z)$$ is represented by $\id$.
Here the map $H_1(F\times 0;\Z)\to H_1(F\times 1;\Z)$ corresponds to the identification of $F\times 0=F\times 1$ via the gluing.

Now consider the following commutative diagram of exact sequences given by the Mayer-Vietoris sequence
\[
\xymatrix@C0.7cm@R0.8cm{  \Lambda \otimes H_1(F;\Z)\ar[d]^=\ar[rr]^{t\iota_+-\iota_-}
&&\Lambda \otimes H_1(Y;\Z)\ar[d]^=\ar[r] &H_1(M;\Lambda)\ar[r]\ar[d]^=&0\\
\Lambda^k\ar[rr]^{tP-\id}&&\Lambda^k\ar[r]\ar[ur]^-{=:\Phi}&H_1(M;\Lambda)\ar[r]&0.}\]
The map $\Phi\colon \Lambda^k\to H_1(M;\Lambda)$ descends to an isomorphism
 $\Lambda^k/(tP-\id)\to H_1(M;\Lambda)$. We claim that this isomorphism is the desired isometry. This follows immediately from the following calculation, which builds on Theorem~\ref{thm:computebl-intro}:
\[
\ba{rcl}
\Bl(d_i,d_j)= \Bl(\iota_-(c_i),\iota_-(c_j))&=&-(\iota_+-t^{-1}\iota_-)^{-1}(\iota_+(c_i))\underset{F}{\cdot} c_j\\
&=&\mbox{$ij$-entry of }-\big((P-t^{-1}\id)^{-1}P\big)^T J\\
&=&\mbox{$ij$-entry of }-P^T(P^T-t^{-1}\id)^{-1}J\\
&=&\mbox{$ij$-entry of }(-\id+t^{-1}(P^T)^{-1})^{-1}J\\
&=&\mbox{$ij$-entry of }J(t^{-1}P-\id).\ea
\]
Here in the last equality we used that $J=P^TJP$, i.e.\ that $(P^T)^{-1}J=JP$.
\end{proof}

\end{document}